\documentclass[a4, 12pt]{amsart}
\usepackage{amssymb}
\usepackage{amstext}
\usepackage{amsmath}
\usepackage{amscd}
\usepackage{latexsym}
\usepackage{amsfonts}
\usepackage{color}
\usepackage{enumerate}
\usepackage{textcomp}
\usepackage[all]{xy}
\usepackage{graphicx}

\theoremstyle{plain}
\newtheorem{thm}{Theorem}[section]
\newtheorem*{thm*}{Theorem}
\newtheorem*{cor*}{Corollary}
\newtheorem*{defn*}{Definition}

\newtheorem{lem}[thm]{Lemma}
\newtheorem{cor}[thm]{Corollary}

\newtheorem*{claim*}{Claim}

\theoremstyle{definition}
\newtheorem{defn}[thm]{Definition}
\newtheorem{ex}[thm]{Example}
\newtheorem{rem}[thm]{Remark}

\theoremstyle{remark}

\numberwithin{equation}{thm}

\def\a{\mathfrak a}

\def\m{\mathfrak m}

\def\fn{\mathbb N}

\newcommand{\calD}{\mathcal{D}}

\def\a_i{\underline {a_i}}

\def\Ass{\mathrm{Ass}}

\def\supp{\mathrm{Supp}}

\def\Irr{\mathrm{Irr}}

\def\card{\mathrm{card}}
\def\astab{\mathrm{astab}}
\def\dstab{\mathrm{dstab}}

\begin{document}

\author{ Marcel
Morales}
\address{Universit\'e  Grenoble Alpes, Institut Fourier, 
UMR 5582, B.P.74,
38402 Saint-Martin D'H\`eres Cedex,
(FRANCE)}
\email{ marcel.morales@univ-grenoble-alpes.fr}

\author{Nguyen Thi Dung}
\address{Thai Nguyen University of Agriculture and Forestry, Thai Nguyen, Vietnam}
\email{nguyenthidung@tuaf.edu.vn}
\title[Make a graph factor-critical, dstab(G), astab(G)]{Factor-critical graphs and dstab, astab for an edge ideal}
\thanks{  This research was supported partially by Vietnam National Foundation for Science and Technology Development 
(NAFOSTED) under grant number 101.04-2017.14, the project under grant number T2019-01-DH, Institut Fourier, Grenoble, France and VIASM Vietnam}
\begin{abstract} Let $G$ be a simple, connected non bipartite graph and let $I_G$ be the edge ideal
of $G$. In our previous work we showed that L. Lov\'{a}sz's theorem on ear decompositions of
factor-critical graphs and the canonical decomposition of a graph given by Edmonds and Gallai are basic tools for the irreducible decomposition of $I^{k}_G$. In this paper  we use some tools from graph theory, mainly Withney's theorem on ear decompositions of 2-edge connected graphs in order to introduce a new method to make a  graph  factor-critical. We can describe the set $\cup_ {k=1}^{\infty}\Ass (I^{k}_G) $   in terms of some subsets of $G$. We give  explicit formulas for the numbers $\astab(I_G)$ and  $\dstab(I_G)$, which are, respectively, the smallest number  $k$ such that  $\Ass (I^{k}_G)=\Ass (I^{k+i}_G)$ for all $i\geq 0$  and the smallest number $k$  such that the maximal ideal belongs to $ \Ass (I^{k}_G)$. We also give very simple upper bounds for  $\astab(I_G)$ and $\dstab(I_G)$.
\end{abstract}
\maketitle
\tableofcontents

\section{Introduction}  In this paper we will study the sets of associated  prime ideals
of the powers of the edge ideal $I_G\subset R:=K[x_1,\ldots,x_d]$ of a graph $G=(V,E)$ with $V=\{P_1,\ldots,P_d\}$. There are many articles on this topic, such as \cite{B},\cite{CMS},\cite{FHV2},\cite{HM},\cite{MMV},\cite{MD} and \cite{T}.
In \cite{B}, Brodmann showed that when $R$ is a Noetherian ring and $I$ is an ideal
of $R$, the sets $\Ass (I^{k}_G)$ stabilize for large $k$. In \cite{CMS} a method is given to construct   associated primes of the powers of an edge ideal is given. 
In \cite{MMV} it was proved that for edge ideals   the sets of associated primes of the powers of $I$ form an ascending chain, known as the persistence property. In \cite{MD} we proved the strong persistence property, that is an irreducible primary component of a power  $I^{k}_G$ induces an irreducible primary component in $I^{k+1}_G$ with the same associated prime.
In \cite{MD} we describe the set of irreducible primary components of $I^{k}_G$ in terms of factor-critical sets and the canonical decomposition of a graph given by Edmonds and Gallai (see for example \cite{E}, \cite{G1}, \cite{G2},\cite{L},\cite{L1},\cite{LP}).\\
 Ear decomposition is an important tool in our work (see \cite[Chapter 5]{YL}); the key result is Withney's theorem which states that a graph if 2-edge connected if and only if it has an ear decomposition. \\ 
In Section 2 we introduce the work of A.Frank \cite{F} and Z.Szigeti \cite{S} (see also \cite[Chapter 5]{YL}), who describe the optimal way to make a 2-edge connected graph factor-critical. We  extend their work by introducing a generalized ear decomposition for a connected graph and  two invariants $\varphi (G)$, the minimum number of even ears in a generalized ear decomposition, and $\psi (G)$   the number of bridges. \\
 In Section 3, inspired by  A.Frank \cite{F}, Z.Szigeti \cite{S} and  Lemma 4.7 in our paper \cite{MD}, we can describe a new method using duplication of vertices in a non bipartite graph (after A. Schrijver \cite{Sc}) and a generalized ear decomposition to obtain factor-critical graphs with the minimum number of steps. \\ 
In Section 4, by using the description of  the irreducible primary components of $I_G^k$ in terms of matching-critical graphs  given in \cite{MD} and  the method developed in Section 3, we can  describe the set of associated primes of each power of the edge ideal $I_G$. Our result extends Proposition 3.3 of \cite{MMV}, where the case of $\m^{ {\bf 1 }_V}$ was considered, and theorem 4.1 of \cite{CMS}. As an application we compute exactly the smallest number number  $k$, called $\astab(I_G)$, where $\Ass (I^{k}_G)=\Ass (I^{k+i}_G)$ for all $i\geq 0$, and the smallest number number  $k$, called  $\dstab(I_G)$,  where the maximal ideal belongs to $ \Ass (I^{k}_G)$.\\
To be more precise let us introduce some notations and definitions. For any vector   ${\bf a }=(a_1,\ldots,a_d)\in \fn^d$ we set $\m ^{{\bf a}}$ the monomial ideal generated by $x_i^{a_i}$. For any subset $U\subset V$ we set ${{\bf 1}_{U}}$  the vector whose $i-$coordinate is 1 if $P_i\in U$ and $0$ otherwise. We will denote by $G_U$ the induced subgraph of $G$ with vertices in $U$.
 Our first result is:
 \begin{thm*}\label{th-astab-primes} Let $G$ be a simple, connected non bipartite graph. Then $\m ^{{\bf 1}_{U}} \in \Ass (I^{\astab(I_G)}_G) $ if and only if $Z:= V\setminus  U$ is a coclique set and either $U=N(Z)$ or $U\not=N(Z)$   and every connected component of $G_{ U\setminus N(Z)}$ contains an odd cycle.
 \end{thm*}
\begin{defn*}Let $G=(V,E)$ be a simple, connected non bipartite graph. Let $U\subset V, Z=V\setminus U$. If $Z$ is not an independent set we set  $ \calD^*(U)= \emptyset $. From now on we assume that  $Z$ is an independent set. A set $W\subset V\setminus (Z\cup N(Z))$ is called {\it $U$-dominant}  
if $V=N(W)\cup N(Z)\cup Z$.  $W$ is called $U$-dominant* if either $W=\emptyset$ or it is $U$-dominant and every connected component of $G_W$ contains an odd cycle.  We will denote by $\calD^*(U)$ the set of all subsets $W\subset U$ that are $U$-dominant*.\\
For every $W\in \calD^*(U) $  we define $\nu ^*(G_W)= \frac{1}{2}(\card(W)+\varphi (G_W)+\psi (G_W)-t_W)$,  where   $t_W$ is the number of connected components of $G_W$. 
 \end{defn*}
First we describe $\Ass (I^{k}_G) $ for every $k\geq 1$.
\begin{thm*}\label{th-astab-intro} Let $G$ be a simple, connected non bipartite graph, then
$$\Ass (I^{k}_G) =\{ \m ^{{\bf 1}_{U}} \mid U\subset V, \calD^*(U)\not=\emptyset,\min_{ W\in \calD^*(U) }\nu ^*(G_W)\leq   k-1  \}.$$
 \end{thm*}
 Then we give explicit formulas for $\astab(I_G), \dstab(I_G)$.
 \begin{thm*}\label{astab-main-intro} Let $G$ be a simple, connected non bipartite graph. We have 
$$ \astab(I_G)=\max_{\{U\subset V \mid  \calD^*(U)\not=\emptyset \}}\{\min \{ 1+\nu ^*(G_W) \mid W \in \calD^*(U)\}  \}.$$
$$ \dstab(I_G)=\min \{ 1+\nu ^*(G_W)\mid W \in \calD^*(V)\}.$$ 
\end{thm*}
Then we can derive a very simple upper bound for $\astab(I_G),\astab(I_G)$:
\begin{thm*}\label{bound-a-d-stab-intro} Let $G$ be a simple connected non bipartite graph. 
$$ \dstab(I_G) \leq \frac{1}{2} (\card(E(G)) + \psi(G)) + 1,$$
$$ \astab(I_G)\leq \card(E(G))-k+1,$$
where $2k-1$ is the minimum length of an odd cycle in $G$.
 \end{thm*}
We then  give an example of a graph with $2l+3$ vertices such that $\dstab(I_G)=l+1, \astab(I_G)=2l-2 $, proving that the difference between $\dstab(I_G), \astab(I_G)$ can be as large as possible.\\
Finally, in Section 5 we prove that every 2-edge connected non bipartite graph $G$ has an ear decomposition, with $\varphi (G)$ even ears, starting with an odd cycle.\\
Recall that the problem of determining dominant sets or independent sets in a graph is a NP-hard, but can be done by hand for a graph with small number of vertices.
\section{Preliminaries}  
Given a non-factor-critical
graph, how can we turn it into a factor-critical graph? How can we measure how
far a graph is from a factor-critical graph?  In \cite[Chapter 5]{YL} the authors consider
 three possible operations on 2-edge-connected graphs: contraction, subdivision and ear
decomposition. In this section we introduce these operations and extend  the results of Szigeti \cite{S}, who  showed that for  2-edge-connected graphs, all the three approaches  are  equivalent.

\begin{defn} (\cite {G3}) {\it An ear decomposition} $G_0, G_1, \ldots, G_k=G$ of a graph $G$ is a sequence of graphs where the
first graph $G_0$ being  a vertex, edge, even cycle, or odd cycle, and each graph $G_{i+1}$ is  obtained from $G_i$ by adding an ear.

Adding an ear is done as follows: take two vertices $a$ and $b$ of $G_i$ and add a path $F_i$ from $a$ to $b$ such that all vertices on the path except $a$ and $b$ are new vertices (present in $G_{i+1}$ but not in $G_i$). An ear with $a\ne b$ is called {\it  open}, otherwise, {\it closed}. An ear with $F_i$ having an odd (even) number of edges is called {\it odd (even)}.
The sequence of ears $F_0, F_1,\ldots , F_s,$ is also called an ear decomposition.
\end{defn}

From Withney's theorems, (see \cite[Chapter 5]{YL}), every 2-edge-connected graph has an ear
decomposition and by Lov\'{a}sz's   theorem \cite{L} a graph is factor-critical if and only if it has an  ear decomposition with odd ears.

\begin{thm}  \label{t-Withney} (Withney) A graph $G = (V, E)$ is 2-edge-connected if and only if has an ear decomposition.
 Furthermore $G$ can be constructed from any vertex $P$ by successively adding an ear to the  previously constructed graphs.
\end{thm}
Note that 2-edge-connected  means bridgeless.
Since any connected graph $G$ can be written as a sequence of bridgeless graphs and bridges,  we immediately have the following result.
\begin{thm} {\bf and Definition.}
Every connected graph $G$ has a generalized ear decomposition, that is:  $F_0, F_1,\ldots , F_s,$ such that $F_0$ is a cycle or a bridge, and $F_i$ is either an ear, or a bridge. If $G$ is not a tree $F_0$ can be chosen to be a cycle, and   if $G$ is a non bipartite graph, $F_0$ can be chosen to be an odd cycle.
 \end{thm}
 Let $\varphi (G)$ denote  the minimum  number of even ears in  generalized ear decompositions of $G$, note that if $G_1,G_2,\ldots,G_l$ are the 2-edge connected components of $G$, then  $\varphi (G)=\sum_{i=1} ^l \varphi (G_i)$. A generalized ear decomposition of $G$ with $\varphi (G)$ even ears is called optimal. The number of bridges  is an invariant, denoted by $\psi (G)$.  Clearly, a connected graph $G$ is factor-critical if and only if $\varphi( G)=\psi (G)=0$.
 \begin{lem} \label{number-ears}
The number of ears in a generalized ear decomposition of a connected graph $G$ is independent of the ear decomposition and is equal to     $\card(E(G))-\card(V(G))+ \psi(G)+1$.
 \end{lem}
 \begin{proof}  The proof follows by induction on the number of bridges in $G$. Our claim is well known for 2-edge connected graphs. Suppose that $G$ has at least one bridge. Let $F_0, F_1,\ldots , F_s,$ be a generalized ear decomposition of $G$. Note that if $F_i$ is a bridge for some $i<s$ then the graph $G'$ union of $F_{i+1}, F_{i+2},\ldots , F_{s}$ has  $F_{i+1}, F_{i+2},\ldots , F_{s}$ as a generalized ear decomposition. Let $H$ be the graph with the generalized ear decomposition  $F_0, F_1,\ldots , F_{s-1}.$ If $F_s$ is a bridge then $\psi (H)=\psi (G)-1$, $\card(E(H))=\card(E(G))-1,\card(V(H))=\card(V(G))-1,$ by induction hypothesis we have 
$s=\card(E(H))-\card(V(H))+ \psi(H)+1$ hence by a simple calculation we get $s+1=\card(E(G))-\card(V(G))+ \psi(G)+1$. Similar arguments apply when $F_0$ is a bridge.
If $F_0, F_s$ are not  bridges then there exists $0<i<s$ such that $F_i$ is a bridge. Let $H_1$ be the graph with the generalized ear decomposition $F_0, F_1,\ldots , F_{i-1}$ and $H_2$ be the graph with the generalized ear decomposition $F_{i+1}, F_{i+2},\ldots , F_{s}$. We have  
  $\psi (H_1)+\psi (H_2)=\psi (G)-1$, $\card(E(H_1))+\card(E(H_2))=\card(E(G))-1,\card(V(H_1))+\card(V(H_2))=\card(V(G)).$ by induction hypothesis we have 
$\card(E(H_1))-\card(V(H_1))+ \psi(H_1)+1=i, \card(E(H_2))-\card(V(H_2))+ \psi(H_2)+1=s-i$ hence by a simple calculation we get $s+1=\card(E(G))-\card(V(G))+ \psi(G)+1$.
 \end{proof}
\begin{ex} \label{ex.1}
Let $G=(V,E)$ be a graph with the vertex set $V=\{a,b,c,d,e,f,g,h\}$ and $E(G)=\{ab,bc,ca,cd,de,ef,fg,gd,eh\}$. We have a generalized ear decomposition of $G$:
$$F_0=abc, F_1=cd, F_2=defgd, F_3=eh,$$
where $F_0$ is an odd ear, $F_1, F_3$ are bridges,  $F_2$ is an even ear. So we have $\varphi(G)=1, \psi(G)=2$.
\begin{figure}[!ht]
\centering
\includegraphics[height=6cm,width=8cm]{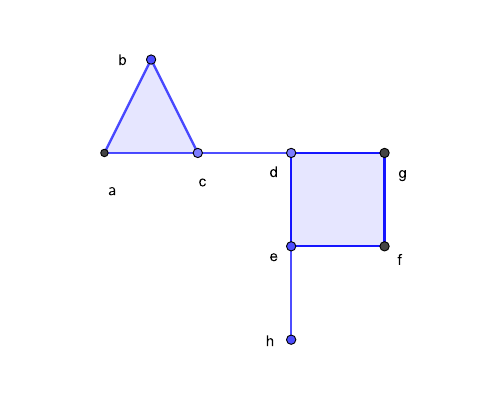}
\caption{$G$ and a generalized ear decomposition  } \label{Fig.1}
\end{figure}
\end{ex}
\begin{defn} The {\it contraction} of an edge $e =uv$ consists of identifying $u$ and $v$ and deleting the edge $e$.
The {\it contraction} of an edge set $F$ of $G$ denoted  $G / F$,  is obtained from $G$ by
contracting each edge of $F$.  For a connected graph $G$  an
edge set $F$ is called {\it critical-making} if the contraction $G/ F$ results in a factor-critical
graph.
\end{defn}

A minimal critical-making set can therefore be used as a measure 
of how close a graph is to factor criticality.

\begin{ex}\label{ex.2}
Let $G=(V,E)$ be a graph with the vertex set $V=\{a,b,c,d,e,f,g,h\}$ and $E=\{ab,bc,ca,cd,de,ef,fg,gd,eh\}$ (see example \ref{ex.1}).
By contracting the edges $F=\{cd,dg,eh\}$  we have a factor-critical graph in figure \ref{Fig.2}.
\begin{figure}[!ht]
\centering
\includegraphics[height=6cm,width=8cm]{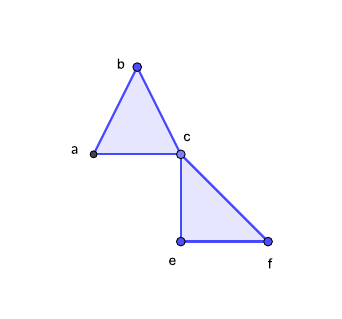}
\caption{$G/F$ contraction of the edges $F=\{cd,dg,eh\}$ } \label{Fig.2}
\end{figure}
\end{ex}

\begin{defn} The {\it subdivision} (or substitution) of an edge   $e=uv$ in a cycle of a graph $G$ means that we add a  new vertex $w$ and replace the edge $e$ by two edges $uw,wv$.The {\it subdivision} of a bridge $e=uv$  of a graph $G$ means that we add  a  new vertex $w$ and two edges $uw,wv$.   The {\it subdivision} of an edge set $F$ of a graph $G$ means that we subdivide each edge $e$ of $F$. The resulting graph is denoted by $G \succ F$.
\end{defn}

Note that our definition of subdivision is a slight different from the usual one.

\begin{ex} \label{ex.3}
Let $G=(V,E)$ be a graph in the example \ref{ex.2}.
By subdividing the edges $F=\{cd,dg,eh\}$ we have a factor-critical graph in figure \ref{Fig.3}.
\begin{figure}[!ht]
\centering
\includegraphics[height=8cm,width=10cm]{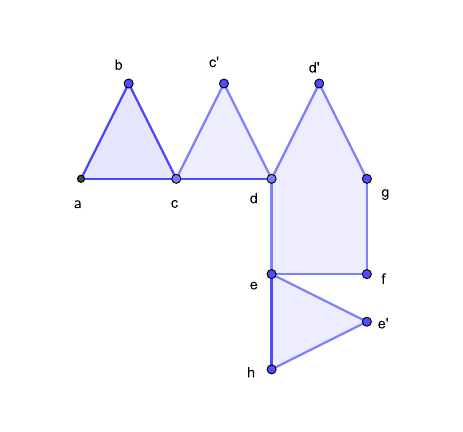}
\caption{$G\succ F$ Subdivision of the edges $F=\{cd,dg,eh\}$   }\label{Fig.3}
\end{figure}
\end{ex}

Szigeti \cite{S} showed that for  2-edge-connected graphs, all the three approaches defined above are actually equivalent.

\begin{thm}(Szigeti \cite{S}) Let  $G$ be a 2-edge connected graph and $k$
be a positive integer. Then the following statements are equivalent:
\begin{enumerate}
\item  $\varphi( G)=k .$
\item $\min \{\card(F)\mid F \text{ is  a  critical-making  set}\} = k.$
\item $\min\{\card(F)\mid  G \succ F\text{ is  a  factor-critical  graph}\} = k.$
\end{enumerate}
\end{thm}

Since every bridge of $G$ is contained in every critical-making set, we have

\begin{thm}Let  $G$ be a connected graph and $k$ be  a positive integer. Then the following statements are equivalent:
\begin{enumerate}
\item  $\varphi( G)+\psi(G)=k .$
\item $\min \{\card(F)\mid F \text{ is  a  critical-making set} \} = k.$
\item $\min\{\card(F)\mid  G \succ F \text{ is a  factor-critical  graph} \} = k.$
\end{enumerate}
\end{thm}
\section{Constructing Factor-critical graphs by replication}
In this section we develop a method to construct factor-critical graphs from a simple connected non bipartite graph in a minimum number of steps by using replication of vertices.

\begin{defn}\label{df-replication} (Schrijver \cite{Sc}) Let $G = (V, E)$ be a graph with  $V=\{P_1,\ldots,P_d\}$ and  ${\bf a }\in \fn^d$  a non-zero vector.  Set
 $A_i=\{P_i=P_i^{(1)},\ldots,P_i^{(a_i)}\}$ for each  $a_i>0$.
The graph  $S:=p_{\bf a }(G)$ with  the vertex set $V(S)=\cup_{ a_i>0 }A_i$ and the edge set  $E(S)=\{P_i^{(l)}P_j^{(m)} \mid P_i\in A_i, P_j\in A_j, P_iP_j \in E\}$  is called {\it the replication (or duplication)} of  $G$ by the vector  ${\bf a }$.     The  {\it support} of $S$ is  the set $\supp (S):=V(S)\cap V=\supp ({\bf a })$. We denote $N_G(S)=N(V(S)\cap V)$.  
For small values of $a_i\leq 3$  sometimes we will write  $P_i,P_i',P_i''$ instead of $P_i^{(1)},P_i^{(2)},P_i^{(3)}$. As usual we set ${\lvert {\bf a } \rvert} =\sum_i a_i.$
\end{defn}

\begin{defn}\label{df-phi} Let $G$ be simple connected  graph. It is well known that $G$ is {\it non bipartite} if and only if it contains an odd cycle. Let $\varphi_{\rm odd} (G)$ denote  the minimum  number of even ears in  generalized ear decompositions of $G$ such that the first ear is an odd cycle.	
\end{defn} 

\medskip
Clearly, a connected graph $G$ is factor-critical if and only if $\varphi_{\rm odd} ( G)=\psi (G)=0$.

\begin{rem} Let $G$ be a 2-edge connected non bipartite  graph. In Section 5 we will prove that $\varphi_{\rm odd} (G)=\varphi(G)$. The proof uses  the works in graph theory of Frank \cite{F} and Szigeti \cite{S}. From now on we write $\varphi(G)$ instead $\varphi_{\rm odd} (G)$.
\end{rem}
In Lemma 4.7 of  \cite{MD} we have   developed a method for extending  factor-critical graphs. We can apply this method to an ear decomposition starting with an odd cycle, since an odd cycle is factor-critical.
\begin{lem}\label{l-makefc_min} Let $G$ be a connected non bipartite  graph. There exists a vector ${\bf a}_G$ with support in $G$ such that 
$p_{{\bf a}_G+ {\bf 1}_G}(G)$ is factor-critical  and  $\lvert{\bf a}_G \rvert = \varphi  (G)+\psi (G).$ Hence
$\nu(p_{{\bf a}_G+ {\bf 1}_G}(G))= \frac{1}{2}(\card(G)+\varphi  (G)+\psi (G)-1).$ 
\end{lem}

 \begin{proof} 
 Since $G$ is  connected non bipartite  graph,  there exists a generalized ear decomposition of $G $:  $F_0, F_1,\ldots , F_s$, such that $F_0$ is an odd cycle with $\varphi  (G)$ even ears. We set   $G_{i-1}$ be the union of $F_0,\ldots ,F_{i-1}$.\\
For each even ear  $F_i : P_{i,1}, \ldots , P_{i,j_i}$  we duplicate the vertex $P_{i,1}$,  let  $Q_{i,1}\in G_{i-1}$ a neighbor of $P_{i,1}$ and let $F'_i$ be the path $F'_i : Q_{i,1},P'_{i,1},P_{i,2}, \ldots , P_{i,j_i}$.\\ 
For each bridge   $F_i : P_{i,1}, P_{i,2} $, with $P_{i,1}\in F_{i-1}$, we duplicate the vertex $P_{i,1}$. Let $Q_{i,1}\in G_{i-1}$ a neighbor of $P_{i,1}$ and let $F'_i$ be the path $F'_i : Q_{i,1},P'_{i,1},P_{i,2},P_{i,1}$.
If $F_i$ is an odd ear we set $F'_i=F_i$. For each vertex $P$ in $G$, let $a_P$ be the number of times minus one that $P$ is duplicated by the above construction, and $a_P=0$ if it no duplicated. Let ${\bf a}_G$ be the vector with support in $G$ with coordinates $a_P, P\in G$. Then $p_{{\bf a}_G+ {\bf 1}_G}(G)$ has an odd ear decomposition
$F'_0, F'_1,\ldots , F'_t$. In particular $p_{{\bf a}_G+ {\bf 1}_G}(G)$ is factor-critical. Moreover  
$$\lvert{\bf a}_G\rvert+ \card(G)=\card(p_{{\bf a}_G+ {\bf 1}_G}(G))= \card(G)+\varphi  (G)+\psi (G),$$ so that $\lvert{\bf a}_G\rvert=\varphi  (G)+\psi (G)$. Hence
  $$\nu(p_{{\bf a}_G+ {\bf 1}_G}(G))= \frac{1}{2}(\card(G)+\varphi  (G)+\psi (G)-1).$$
\end{proof}

\begin{ex}
Let $G=(V,E)$ be a graph with the vertex set $V=\{a,b,c,d,e,f,g,h\}$ as in the figure \ref{Fig.1}. Then we have generalized ear decomposition of $G$: 
$$F_0=abc, F_1=cd, F_2=defgd, F_3=eh,$$
where $F_0$ is odd ear, $F_1, F_3$ are bridges,  and $F_2$ is even ear. Therefore, we have that $ \varphi(G)=1$ and $\psi(G)=2$. Choose the vector ${\bf a}_G=(0,0,1,1,1,0,0,0)\in \mathbb N^8$ and put 
$S=p_{{\bf a}_G+{\bf 1}_G}(G)$. Then we have $V(S)=D(S)=\{a,b,c,c',d',e',f,g,h\}$ and $S$ has an ear decomposition
$$F'_0=abc, F'_1=bc'dc,F'_2=c'd'gd, F'_3=defg, ehe'f,$$
where they are all odd ears. It is clear that 
$$\nu(S)= \frac{1}{2}(\card(G)+\varphi  (G)+\psi (G)-1)= \frac{1}{2}(8+1+2-1)=5.$$ 
\begin{figure}[!ht]
\centering
\includegraphics[height=8cm,width=10cm]{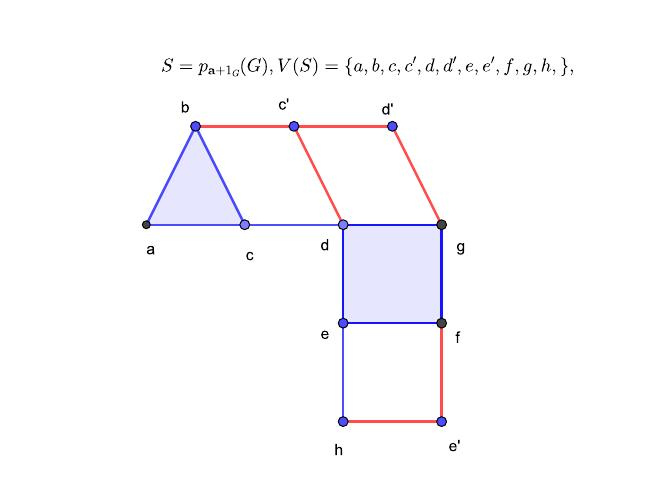}
\caption{$G$ and $p_{{\bf a}+ {\bf 1}_G}(G)$ } \label{Fig.4}
\end{figure}
\end{ex}

\begin{lem}\label{l-makefc}  Let $G$ be a connected non bipartite  graph. If for some vector $\bf a$,  $p_{{\bf a}+ {\bf 1}_G}(G)$ is  factor-critical then 
 $\lvert{\bf a}\rvert\geq \varphi  (G)+\psi (G).$ As a consequence $$\nu(p_{{\bf a}+ {\bf 1}_G}(G))\geq  \frac{1}{2}(\card(G)+\varphi  (G)+\psi (G)-1).$$
 \end{lem}
 
 \begin{figure}[!ht]
\centering
\includegraphics[height=13cm,width=11cm]{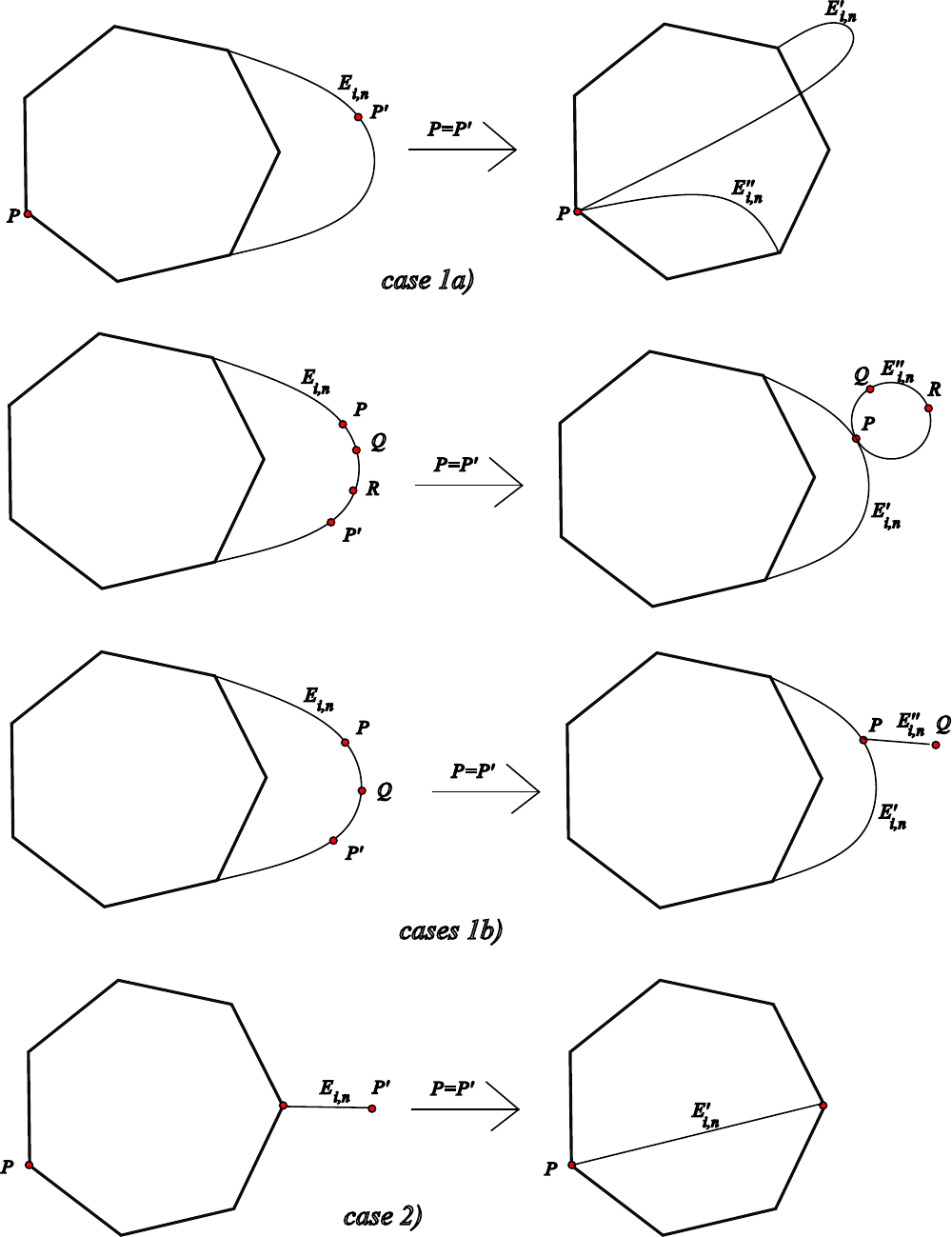}
\caption{Identification of a duplicated vertex } \label{Fig.5}
\end{figure}

\begin{proof}
Let $p_{{\bf a}+ {\bf 1}_G}(G)$ be  factor-critical for some vector $\bf a$. For all $n=0,\ldots,\lvert{\bf a}\rvert,$ we will define by induction a vector 
${\bf a}_n$ such that $\supp({\bf a}_{n})\subset \supp({\bf a})$, $\lvert{\bf a}_{n}\rvert=\lvert{\bf a}\rvert-n$, $p_{{\bf a}_n+ {\bf 1}_G}(G)$ has a generalized ear decomposition and $n\geq (\varphi  +\psi )(p_{{\bf a}_n+ {\bf 1}_G}(G))$. \\ 
We set ${\bf a}_{0}={\bf a}$, since $p_{{\bf a}+ {\bf 1}_G}(G)$ is  factor-critical we have  $F\colon F_0, F_1,\ldots , F_s$ an odd ear decomposition and 
$0= (\varphi  +\psi )(p_{{\bf a}_0+ {\bf 1}_G}(G))$. \\
Let $n\geq 0$ and suppose that we have a vector 
${\bf a}_n$ such that $\supp({\bf a}_{n})\subset \supp({\bf a})$, $\lvert{\bf a}_{n}\rvert=\lvert{\bf a}\rvert-n$, $p_{{\bf a}_n+ {\bf 1}_G}(G)$ has a generalized ear decomposition  $E_{0,n},\ldots, E_{t_n,n}$ and $n\geq (\varphi  +\psi )(p_{{\bf a}_n+ {\bf 1}_G}(G))$. If $n=\lvert{\bf a}\rvert$ our claim is done. So suppose $n<\lvert{\bf a}\rvert$.  Note that in a generalized ear decomposition we can always have $V(E_{0,n})\subset V$. Let     $E_{i,n} \colon  P_{i,1}, \ldots , P_{i,j_i}$ be  a generalized ear  such that all vertices in 
 $E_{0,n}\cup E_{1,n}\cup\dots \cup E_ {i-1,n}$ belong to $V$ but there exists $P'\in V(E_{i,n})\setminus V$ which is the duplicated vertex of $P\in E_{0,n}\cup E_{1,n}\cup\dots \cup E_ {i,n}$. By identification of $P$ and $P'$ we get the vector ${\bf a}_{n+1}={\bf a}_{n}-{\bf 1}_{\{P\}}$. Ears or bridges $E_ {j,n}$ for $j>i$ beginning in $P'$ will begin in $P$. The ear or bridge $E_ {i,n}$ will be modified,  we have several cases
   
\begin{enumerate}
\item $E_{i,n}$ is an ear with several edges.
\begin{enumerate}
	 
\item If  $P'=P_{i,l}$, for some $l\geq 2$  is some internal vertex of $E_{i,n}$ and is the duplication of $P\in E_{0,n}\cup E_{1,n}\cup\dots \cup E_ {i-1,n}$ then let 
$E'_{i,n}\colon P_{i,1}, \ldots,P,$ and $E''_{i,n}\colon  P,P_{i,l+1},\ldots, P_{i,j_i}$.  We replace the ear $E_{i,n}$ by the two ears $E'_{i,n},E''_{i,n}$, ears or bridges ending in $P'$ will end in $P$, so that we get an ear decomposition of $p_{{\bf a}_{n+1}+ {\bf 1}_G}(G)$. Note that the number of edges in $E'_{i,n},E''_{i,n}$ equals the number of edges in $E_{i,n}$. If $E_{i,n}$ is an odd ear then one of $E'_{i,n},E''_{i,n}$ is an even ear and the other is odd so that 
 $n+1\geq (\varphi  +\psi )(p_{{\bf a}_{n+1}+ {\bf 1}_G}(G))=(\varphi  +\psi )(p_{{\bf a}_{n}+ {\bf 1}_G}(G))+1$.
  If $E_{i,n}$ is an even ear then both  $E'_{i,n},E''_{i,n}$ are   even or odd  ears. If both ears are odd then 
 $(\varphi  +\psi )(p_{{\bf a}_{n+1}+ {\bf 1}_G}(G))=(\varphi  +\psi )(p_{{\bf a}_{n}+ {\bf 1}_G}(G)-1$. 
If both ears are even then  $(\varphi  +\psi )(p_{{\bf a}_{n+1}+ {\bf 1}_G}(G))=(\varphi  +\psi )(p_{{\bf a}_{n}+ {\bf 1}_G}(G))+1$. In both cases we have 
 $n+1\geq (\varphi  +\psi )(p_{{\bf a}_{n+1}+ {\bf 1}_G}(G))$.
\item  If  $P'=P_{i,m}$, for some $m\geq 2$  is some internal vertex of $E_{i,n}$ and  $P=P_{i,l}\in E_{i,n}$ then  we can suppose $l<m$. We have the following cases:\\
  If $m=l+2$ then  we get the odd path $E'_{i,n}\colon P_{i,1}, \ldots ,P_{i,l},P_{i,l+3},\ldots , P_{i,j_i}$  and the edge $E''_{i,n}\colon P_{i,l} P_{i,l+1}$. As above we replace the ear $E_{i,n}$ by the two ears $E'_{i,n},E''_{i,n}$, ears or bridges ending in $P'$ will end in $P$,  so that we get a ear decomposition of $p_{{\bf a}_{n+1}+ {\bf 1}_G}(G)$ and we still have $n+1\geq (\varphi  +\psi )(p_{{\bf a}_{n+1}+ {\bf 1}_G}(G))$. \\
If $m>l+2$   then  we have the   path $E'_{i,n}\colon  P_{i,1}, \ldots ,P_{i,l},P_{i,m+1},\ldots , P_{i,j_i}$ and  $E''_{i,n}\colon  P_{i,l}, \ldots ,P_{i,m-1},P_{i,l} $. As above we replace the ear $E_{i,n}$ by the two ears $E'_{i,n},E''_{i,n}$,   ears or bridges ending in $P'$ will end in $P$, other ears are unchanged,  so that we get a generalized ear decomposition of $p_{{\bf a}_{n+1}+ {\bf 1}_G}(G)$ and we still have $n+1\geq (\varphi  +\psi )(p_{{\bf a}_{n+1}+ {\bf 1}_G}(G))$. 
\end{enumerate}

\item $E_{i,n}\colon P_{i,1}, P_{i,2}$ is a bridge with $P_{i,2}\notin E_{0,n}\cup E_{1,n}\cup\dots \cup E_ {i-1,n}$. So that $P'=P_{i,2}$ and  $P\in  E_{0,n}\cup E_{1,n}\cup\dots \cup E_ {i-1,n}$. We have that either $P_{i,1}, P$ is an edge in $ E_{0,n}\cup E_{1,n}\cup\dots \cup E_ {i-1,n}$ or a chord  $E'_{i,n}\colon P_{i,1}, P$. In this last case  we replace the ear $E_{i,n}$ by   $E'_{i,n}$. In both cases  we get a generalized ear decomposition of $p_{{\bf a}_{n+1}+ {\bf 1}_G}(G)$ and we have $n+1\geq (\varphi  +\psi )(p_{{\bf a}_{n+1}+ {\bf 1}_G}(G))$. 
 \end{enumerate}
  \end{proof}

\begin{defn}\label{def-makefc}
 Let  $G=G_1\cup \ldots\cup G_t$ be a simple graph, which is a disjoint union of  $t$ connected non bipartite graphs $G_i, i=1,\ldots,t$.  We set $$\nu ^*(G)= \frac{1}{2}(\card(G)+\varphi  (G)+\psi (G)-t).$$
 \end{defn}

 \begin{thm}\label{th-makefc} Let  $G=G_1\cup \ldots\cup G_t$ be a simple graph, which is a disjoint union of  $t$ connected non bipartite graphs $G_i, i=1,\ldots,t$. Then 
$$\min \{\nu(p_{{\bf a}+ {\bf 1}_G}(G) )\mid p_{{\bf a}+ {\bf 1}_G}(G) {\rm \ is\ matching-critical}\} =\nu ^*(G).$$
Let $k\in \fn$. The following statements are equivalent:
\begin{enumerate}
\item  $\varphi ( G)+\psi(G)=k $ 
\item $\min \{\lvert{\bf a}\rvert\ \ \mid p_{{\bf a}+ {\bf 1}_G}(G) {\rm \ is\ matching-critical}\} = k$
\end{enumerate}
\end{thm}

\begin{proof}
Let  $G$ be a connected non bipartite graph, by lemma \ref{l-makefc} we have  
$$\min \{\nu(p_{{\bf a}+ {\bf 1}_G}(G) )\mid p_{{\bf a}+ {\bf 1}_G}(G) {\rm \ is\ factor\ critical}\} = \frac{1}{2}(\card(G)+\varphi  (G)+\psi (G)-1).$$ 
Let  $G=G_1\cup \ldots\cup G_t$ be a simple graph, which is a disjoint union of  $t$ connected non bipartite graphs $G_i, i=1,\ldots,t$. For $i=1,\ldots,t$, by lemma \ref{l-makefc} we have  
$$\min \{\nu(p_{{\bf a}+ {\bf 1}_{G_i}}(G_i) ) \mid p_{{\bf a}+ {\bf 1}_{G_i}}(G_i) {\rm \ is\ factor\ critical}\} =\nu ^*(G_i).$$ 
By  lemma \ref{l-makefc_min} let ${\bf a}_{G_i}$ be a vector such that $ p_{{\bf a}_{G_i}+ {\bf 1}_{G_i}}(G_i)$ is  factor-critical and $\nu(p_{{\bf a}_{G_i}+ {\bf 1}_{G_i}}(G_i))=\nu ^*(G_i)$. Let ${\bf a}_G=\sum_{i=1}^ {t}  {\bf a}_{G_i}$ then  $p_{{\bf a}_G+ {\bf 1}_G}(G)$ is matching-critical,  $\nu(p_{{\bf a}_G+ {\bf 1}_G}(G))=\sum_{i=1}^ {t} \nu ^*(G_i)$. It is clear that 
$$\nu(p_{{\bf a}_G+ {\bf 1}_G}(G))=\min \{\nu(p_{{\bf a}+ {\bf 1}_G}(G) )\mid p_{{\bf a}+ {\bf 1}_G}(G) {\rm \ is\ matching-critical}\},$$
and 
$$\sum_{i=1}^ {t} \nu ^*(G_i)= \frac{1}{2}(\card(G)+\varphi  (G)+\psi (G)-t).$$ 
For a connected non bipartite graph, the proof of the second assertion follows immediately from lemma \ref{l-makefc}. If the graph is disconnected then the proof is similar to the proof of the first assertion.
\end{proof}
\section{Stable associated primes, dstab and astab}

It is known from  \cite{B}  that $\Ass (I^{k}_G)=\Ass (I^{k+1}_G)$ for $k$ large enough and $\Ass (I^{k}_G)$ is called {\it the set of stable associated primes} of $I_G$.  Furthermore if $G$ is a non bipartite graph, then 
$\m ^{{\bf 1}_{V}}\in \Ass (I^{k}_G)$ for $k$ large enough. The {\it persistence property} $\Ass (I^{k}_G)\subset \Ass (I^{k+1}_G)$ for $k\geq 1$ is proved in  \cite{MMV}. Moreover, we have  the {\it strong persistence property} proved by \cite[Theorem 5.14]{MD} that
 every irreducible component of $I^{k}_G$ induces several irreducible components of $I^{k+1}_G$ with the same associated prime. 
 In this section we will describe the sets $\Ass (I^{k}_G) $  for any $k\geq 1$ in terms of some subsets of $V.$ We denote by $\Irr(I_G^k)$ the set of irreducible components of $I_G^k$. Note that an irreducible component of $I_G^k$, whose radical is  $\m^{{\bf 1 }_U}$, can be written as 
 $\m^{{\bf a }+{\bf 1 }_U},$ for some vector ${\bf a }$ with $\supp({\bf a })\subset U$.
 
\begin{defn} Let $G=(V,E)$ be a non bipartite graph.
$$ \dstab(I_G)=\min \{k\mid  \m ^{{\bf 1}_{V}}\in \Ass (I^{k}_G)\}.$$ 
$$ \astab(I_G)=\min  \{k\mid  \Ass (I^{k}_G)=\Ass (I^{k+i}_G), i\geq 1\}.$$ 
\end{defn}

\medskip
Let $U\subset V$ such that $\m^{ {\bf 1 }_U}\in\Ass (I^{k}_G)$ for some $k\geq 1$. It is well known that $\m^{ {\bf 1 }_U}\in\Ass (I_G)$  if and only if $ V\setminus U$ is a maximal coclique set. From now on we are interested in embedded associated primes, so we can  suppose $k\geq 2$. The description of all the embedded irreducible components of $I_G^k$ is given in the following theorems from \cite{MD}.
\begin{thm}\label{treduction1} Let $\m^{{\bf a }+{\bf 1 }_U}$ be an embedded irreducible component of $I^k_G$   
and  $Z=V\setminus U$. Assume that $\supp({\bf a })\cap N(Z)\not=\emptyset.$ Let ${\bf a }={\bf b }+{\bf c } $ with ${\bf b }, {\bf c }\in \fn^d $ 
 the unique decomposition such that $\supp({\bf b })= \supp({\bf a })\setminus N(Z), \supp({\bf c })= \supp({\bf a })\cap  N(Z)$ 
 and  $\delta =\lvert {\bf c } \rvert.$ 
Then $ \supp({\bf b })\cap N(Z)=\emptyset $ and  $\m^{{\bf b }+{\bf 1 }_U}$ is an embedded irreducible component of $I_G ^{k-\delta }$.
\end{thm}

\medskip
Recall the Gallai-Edmonds canonical decomposition of  a graph as given by J. Edmonds \cite{E} and T. Gallai \cite{G1}. For more details see \cite[Theorem 1.5.3]{YL} or \cite[Section 4]{MD}. 
For any simple graph $G$, denote by $D(G)$ the set of all vertices in $G$ which are missed by at least one maximum matching of $G$,  and $A(G)$ the set of vertices in $V-D(G)$ that are adjacent to at least one vertex in $D(G)$. Let $C(G) =V-A(G)-D(G)$. More generally let $S\subset V$ and $G_S$ its induced subgraph of $G$. The induced graphs $A(G_S),D(G_S),C(G_S)$ will be denoted by  $A(S), D(S), C(S)$. The graph  $D(S)$ is  matching-critical if it is a disjoint union of factor-critical graphs.

\begin{thm}\label{t22newbis20} Let $k\geqslant 2$ be an integer, ${\bf a}\in \fn^d$ be a nonzero vector, $U\subset V$ such that 
$\supp({\bf a})\subset U$. Let denote  $Z:=V\setminus U$  and  $p_{\bf a}(G)$  the replication of $G$ by ${\bf a}$. 
Assume that $\supp({\bf a})\cap N(Z)=\emptyset$. Then the following statements are equivalent:
 
(i)  $\m^{{\bf a}+{\bf 1}_U}$ is an embedded irreducible component  of $I^k_G$.
 
(ii) $p_{\bf a}(G)$ and $ Z$ satisfy the following properties 
\begin{enumerate}
\item  $Z$ is  a coclique set, $ \nu(p_{\bf a}(G))= k-1$ and  $V=N_G(D(p_{\bf a}(G)))\cup Z\cup N(Z) $.
\item  $C(p_{\bf a}(G))=\emptyset $, i.e. $p_{\bf a}(G)=D(p_{\bf a}(G))\cup A(p_{\bf a}(G))$ in the Gallai-Edmonds's canonical decomposition.
\end{enumerate}
\end{thm}

\medskip
Moreover in \cite[Section 5]{MD} we also proved that every connected component of $D(p_{\bf a}(G))$  has at least 3 vertices.
 Let  $\m^{ {\bf 1 }_U}\in\Ass (I^{k}_G)$ be an embedded associated prime,  we define $$ \astab(I_G, U)\colon =\min \{k\mid  \m ^{{\bf 1}_{U}}\in \Ass (I^{k}_G)\}.$$ 
Note that $\dstab(I_G)=\astab(I_G, V)$.

\begin{lem}\label{embass}Let $U\subset V$. Then $\m^{ {\bf 1 }_U}\in\Ass (I^{k}_G)$ is an embedded associated prime for some $k\geq 2$ if and only if  
	\begin{enumerate}
		\item $Z=V\setminus U$ is a coclique set.
		\item There exists a vector  ${\bf a}$ such that $ \supp({\bf a })\cap (N(Z)\cup Z)=\emptyset  $, 
		$p_ {\bf a}(G)$ is matching-critical,  $V=N_G(p_ {\bf a}(G))\cup Z\cup N(Z) $ and $\nu(p_ {\bf a}(G))\leq  k-1$.
	\end{enumerate}
  As a consequence we have 
 $$\astab(I_G, U)=\min \{k\mid  \m ^{{\bf a }+{\bf 1}_{U}} \in \Irr( I^{k}_G)\} = \min \{ \nu(p_ {\bf a}(G))+1\}, $$ 
 where the vector ${\bf a}$ satisfies property (2).
  \end{lem}

\begin{proof} Let $U\subset V$  and  $Z=V\setminus U$. Then $ \m^{ {\bf 1 
		}_U}\in\Ass (I^{k}_G)$ if and only if  there exists a vector ${\bf b 
	}$ such that  $\m^{{\bf b}+{\bf 1 }_U}$ is an embedded irreducible 
	component of $I_G ^{k}$. By  Theorem  \ref{treduction1} we can 
	assume that $ \supp({\bf b  })\cap (N(Z)\cup Z)=\emptyset $. By 
	Theorem \ref{t22newbis20}, we have that $Z$ is a coclique set,   
	$V=N_G(p_ {\bf b }(G))\cup Z\cup N(Z) $, $\nu(p_ {\bf b }(G))=k-1$, 
	and $p_{\bf b}(G)=D(p_{\bf b }(G))\cup A(p_{\bf b }(G))$.  The graph 
	$D(p_{\bf b }(G))$ is matching-critical and we have by Gallai-Edmonds   
	decomposition theorem that  $\nu(D(p_{\bf b }(G))= 
	k-1-\card(A(p_{\bf b}(G)))$. Let  ${\bf a}$ be the vector such that   
	$p_{\bf a}(G)=D(p_{\bf b }(G))$. We can see  $p_{\bf a}(G)$ satisfies 
	the conditions of  Theorem \ref{t22newbis20}, so   $\m^{{\bf a }+{\bf 
			1 }_U}$ is an embedded irreducible component of $I_G ^{k'}$, where   
	$k'=k-\card(A(p_{\bf b}(G)))\leq k.$
\end{proof}

\medskip
Now we can describe the associated primes of $I^{k}_G$ for  $k\geq 2$. Our result extends \cite[Proposition 3.3]{MMV}  which considered the case of $\m^{ {\bf 1 }_V}$  and \cite[Theorem 4.1 ]{CMS}.

 \begin{thm}\label{Tembass}
  $\m^{ {\bf 1 }_U}$ is an embedded associated prime of $I^{k}_G$ for some $k$ if and only if $Z:= V\setminus  U$ is a coclique set, eventually the empty set,  $U\not=N(Z)$   and every connected component of $G_{U\setminus N(Z)}$ contains an odd cycle.
   \end{thm}

 \begin{proof}  If  $\m^{ {\bf 1 }_U}$ is an embedded associated prime of $I^{k}_G$ for some $k$, then by Lemma \ref{embass}, $Z$ is a coclique set,  there exists a vector ${\bf a }$ such that $ \supp({\bf a })\cap (N(Z)\cup Z)=\emptyset  $, $p_ {\bf a}(G)$ is matching-critical and $V=N_G(p_ {\bf a}(G))\cup Z\cup N(Z) $. Let $W=\supp({\bf a })$ and $G_W$  the induced subgraph of $G$. Since $p_ {\bf a}(G)$ is matching-critical and every connected component has at least 3 vertices,  then every connected component of $G_W$ has an odd cycle.
 If  $T\subset N(W)\setminus W$ then every connected component of $G_{W\cup T }$ has an odd cycle. Note that    
 $U\setminus N(Z)=N(W)\setminus N(Z) =W\cup T$ for some subset $T\subset N(W)\setminus W$. So every connected component of $G_{U\setminus N(Z)}$ has an odd cycle. Reciprocally if $U\not=N(Z)$   and every connected component of $G_{U\setminus N(Z)}$ contains an odd cycle, then by Lemma \ref{l-makefc}, there exists a vector  ${\bf b },\supp({\bf b })\subset U\setminus N(Z) $ such that 
$p_ {{\bf b}+{\bf 1}_{U\setminus N(Z)}}(G))$ is matching-critical and $V=N_G(p_ {{\bf b}+{\bf 1}_{U\setminus N(Z)}}(G)))\cup Z\cup N(Z).$ Thus  by Theorem  \ref{t22newbis20} we have that $\m^{ {\bf 1 }_U}$ is an embedded associated prime of $I^{k}_G$ for some $k=\nu(p_ {{\bf b}+{\bf 1}_{U\setminus N(Z)}}(G)))+1$.
 \end{proof}
 
 \medskip
Once we know a stable  associated prime $\m^{ {\bf 1 }_U}$,   finding the smallest $k$ such that  $\m^{ {\bf 1 }_U}$ is an associated prime of $I_G^k$ requires a finer analysis of some subsets of $U$.

\begin{defn}\label{d-dominant*}Let $G=(V,E)$ be a simple, connected non bipartite graph. Let $U\subset V$ and $Z=V\setminus U$. If $Z$ is an independent set then a set $W\subset V\setminus (Z\cup N(Z))$ is called {\it $U$-dominant}  
if $V=N(W)\cup N(Z)\cup Z$ and  {\it $U$-dominant*} if either $W=\emptyset$ or it is $U$-dominant and every connected component of $G_W$ contains an odd cycle.  We will denote by $\calD^*(U)$ the set of all subsets $W\subset U$ that are $U$-dominant*. If $Z$ is not an independent set then we set  $ \calD^*(U)= \emptyset $.
 \end{defn}

\begin{rem} With the notation introduced in \ref{d-dominant*}, we have that $\m ^{{\bf 1}_{U}}$ is a associated prime of $I_G$ if and only if $\calD^*(U)=\{ \emptyset  \}$. Theorem \ref{Tembass} can be reformulated as follows: 
 $\m^{ {\bf 1 }_U}$ is an embedded associated prime of $I^{k}_G$ for some $k$ if and only if  $\calD^*(U)\not= \emptyset, \{ \emptyset  \}$ and in this case we have
 $U\setminus  N(Z)\in  \calD^*(U)$.
\end{rem}

In  \ref{def-makefc} we have defined $\nu ^*(G_W)$ for each $\emptyset \not =W\in \calD^*(U)$. If $W= \emptyset$ then we set $\nu ^*(G_W)=0$. 

\begin{ex}\label{ex-astab} 
Let $G=(V,E)$ be a graph with the vertex set $V=\{a,b,c,d,e,f,g,h\}$ as in the figure \ref{Fig.1}. Now we look for sets $U\subset V$ such that $ \calD^*(U)\not= \emptyset$. Note that $Z=V\setminus U$  is an independent set and if $W\not= \emptyset$  is  $U$-dominant*, then 
$G_W$ is connected and contains the cycle with vertices $a,b,c$, since there is a unique odd cycle in $G$. Note also that $N(W)\cap Z=\emptyset $. We can describe the possible cases of the independent set   $Z$  such that $ \calD^*(U)\not= \emptyset$. 
\begin{enumerate}
\item If $\card(Z)=0$ then $ \calD^*(U)\not= \emptyset$ if and only if $U_1=V$ and $W_1=\{a,b,c,d,e\}\in \calD^*(U_1), \nu ^*(G_{W_1})=3$.
\item If $\card(Z)=1$ then $ \calD^*(U)\not= \emptyset$ if and only if $Z_2=\{ h \}$. We have $U_2=\{a,b,c,d,e,f,g\}$ and $W_2=\{a,b,c,d,g\} \in \calD^*(U_2), \nu ^*(G_{W_2})=3$.
\item If $\card(Z)=2$ then $ \calD^*(U)\not= \emptyset$ if and only if $Z_3=\{ e,g \}$, $Z_4=\{ f,h \}$, or $Z_5=\{ g,h \}$ . We have $U_3=\{a,b,c,d,f,h\}$ $U_4=\{a,b,c,d,e,g\}$ or $U_5=\{a,b,c,d,e,f\}$. and $W_3=W_4=W_5=\{a,b,c\} \in \calD^*(U_i), \nu ^*(G_{W_i})=1$, for  $i=3,4,5$.
 \item If $\card(Z)\geq 3$ then $Z$ is maximal, $\calD^*(U)=\{ \emptyset  \}$. 
 \end{enumerate} As a consequence we have:
 \begin{enumerate}
\item $\Ass (I^{2}_G)\setminus \Ass (I_G)= \{\m^{{\bf 1 }_{U_3}}, \m^{{\bf 1 }_{U_4}},\m^{{\bf 1 }_{U_5}}\}$, $\Ass (I^{3}_G)= \Ass (I^{2}_G)$,
\item $\Ass (I^{4}_G)\setminus \Ass (I^{3}_G)= \{\m^{{\bf 1 }_{U_1}},\m^{{\bf 1 }_{U_2}}\}$, $\Ass (I^{i}_G)= \Ass (I^{i+1}_G)$ for $i\geq 4$,
\item $ \astab(I_G)= \dstab(I_G)=4$.
 \end{enumerate}
  \end{ex}

\begin{thm}\label{th-astab} Let $G$ be a simple, connected non bipartite graph, $U\subset V$ such that $\m^{ {\bf 1 }_U}$ is an embedded associated prime of $I^{k}_G$ for some $k$. Then
$$ \astab(I_G, U)=\min \{ 1+\nu ^*(G_W)\mid W \in \calD^*(U) \}.$$
   \end{thm}
 
 \begin{proof} Let $p_{{\bf a}}(G)$ be a matching-critical graph such that $\m^{{\bf a }+{\bf 1 }_U}$ is an embedded irreducible component of $I_G ^{k}$. Let $W=\supp( {\bf a}),$ then   $W\in \calD^*(U) $. Let $W_1,\dots, W_ {t_W}$ be the connected components of $G_W$ and  ${\bf a}_i$  the restriction of  ${\bf a}$ to $W_i$, so that ${\bf a}=\sum_{i=1}^{t_W}{\bf a}_i$,  and $p_{{\bf a}_i}(G)$  is factor-critical. There exist by theorem \ref{th-makefc} the vectors   ${\bf b}_i$ such that $p_{{\bf b}_i+ {\bf 1}_{W_i}}(G)$ is  factor-critical and
$$\nu(p_{{\bf b}_i+ {\bf 1}_{W_i}}(G))= \frac{1}{2}(\card(W_i)+\varphi  (G_{W_i})+\psi (G_{W_i})-1) \leq \nu(p_{{\bf a}_i}(G)).$$
Let ${\bf b}(W)=\sum_{i=1}^{t_W}{\bf b}_i$. It  follows that $p_{{\bf b}(W)+{\bf 1}_W}(G)$ is matching-critical and
$$ \nu (p_{{\bf b}(W)+{\bf 1}_W}(G))=\min \{ \nu (p_{{\bf a}}(G))\mid  W=\supp( {\bf a}), p_{{\bf a}}(G) \text {is \ matching-critical}\}.$$
Note that $\nu(p_{{\bf b}+ {\bf 1}_{W}}(G)=\nu ^*(G_W)$. Hence
$$ \astab(I_G, U)=\min \{ 1+\nu ^*(G_W)\mid W \in \calD^*(U) \}.$$ 
  \end{proof}
\begin{cor}\label{c-astab} Let $G$ be a simple, connected non bipartite graph. Then 
$$\Ass (I^{k}_G) =\{ \m ^{{\bf 1}_{U}} \mid U\subset V, \calD^*(U)\not=\emptyset,\min_{ W\in \calD^*(U) }\nu ^*(G_W)\leq   k-1  \}.$$
   \end{cor}

 \begin{proof} By the definition of  
$\Ass (I^{k}_G) =\{ \m ^{{\bf 1}_{U}} \mid U\subset V, \astab(I_G, U)\leq  k  \}$, we can see that the proof follows from the above theorem.
  \end{proof}

  As an immediate consequence of Theorem \ref{Tembass} and Corollary  \ref{c-astab} we have 
   
\begin{thm}\label{astab-main} Let $G$ be a simple, connected non bipartite graph. We have 
$$ \astab(I_G)=\max_{\{U\subset V \mid \calD^*(U)\not=\emptyset \}}\{\min \{ 1+\nu ^*(G_W) \mid W \in \calD^*(U)\}  \}.$$
$$ \dstab(I_G)=\min \{ 1+\nu ^*(G_W)\mid W \in \calD^*(V)\}.$$ 
\end{thm}

 \begin{proof} We have  $\dstab(I_G)=\astab(I_G, V)$ and  $$ \astab(I_G)=\max_{\{U\subset V \mid  \calD^*(U)\not=\emptyset \}}\{\astab(I_G,U) \}.$$
 \end{proof}

Now we can give upper bounds for $\astab(I_G)$ and $\dstab(I_G, V)$.
 
\begin{thm}\label{bound-a-d-stab} Let $G$ be a simple connected non bipartite graph. Then we have
$$ \dstab(I_G)\leq  \frac{1}{2}( \card(E(G))+\psi (G)-1)+1.$$
$$ \astab(I_G)\leq \max_U \{  \frac{1}{2}(\card(E(G_{U\setminus N(Z)}))+\psi (G_{U\setminus N(Z)})-1)+1\},$$
 where the $\max $ is taken over all the sets $U\subset V(G)$ such that $Z\colon=V\setminus U$ is an independent set,  $U\not=N(Z)$   and every connected component of $G_{U\setminus N(Z)}$ contains an odd cycle.
 \end{thm}

 \begin{proof} By Lemma \ref{number-ears}, for every connected non bipartite graph $H$ we have $$\varphi  (H)+\psi (H)\leq \card(E(H))-\card(V(H))+\psi (H).$$ Hence 
 $$\nu ^*(H)=  \frac{1}{2}(\card(V(H))+\varphi  (H)+\psi (H)-1)\leq  \frac{1}{2}(\card(E(H))+\psi (H)-1).$$ 
 Our claim follows immediately from Theorems \ref{Tembass} and Theorem  \ref{astab-main}.
 \end{proof} 

\medskip
As a consequence we get a  very simple upper bound for $\astab(I_G)$.

\begin{thm}\label{bound-astab} Let $G$ be a simple connected non bipartite graph and $2k-1$  the minimum length of an odd cycle in $G$. Then we have 
$$ \astab(I_G)\leq \card(E(G))-k+1.$$
\end{thm}

\begin{proof} With the notation of  Theorem \ref{bound-a-d-stab}, we have
$$\card(E(G_{U\setminus N(Z)}))\leq \card(E(G)), \psi (G_{U\setminus N(Z)})\leq \card(E(G))-(2k-1).$$ Hence  
$$ \astab(I_G)\leq  \frac{1}{2}(\card(E(G))+\card(E(G))-(2k-1)-1)+1=\card(E(G))-k+1,$$ 
as required.
\end{proof}

\medskip
Let  $G$ be a cycle of the length $2k-1$. Then $\card(E(G))-k+1=k.$ It is well known  that  
$ \astab(I_G)=\dstab(I_G)=k$. Therefore one can see from Theorem \ref{bound-astab}  that our bound is tight.

\begin{ex}
Let $G=(V,E)$ be a graph with the vertex set $V=\{a_1,a_2,\ldots,a_{2n+1}\}$, which is the union of $n$ triangles with a common vertex $a_1$  (See Figure \ref{Fig.6-0}). Apply Theorem \ref{astab-main}, let $U\subset V,U\not= V$ such that $\calD^*(U)\not=\emptyset$ and $\calD^*(U)\not=\{\emptyset\}.$  Note that   $Z=V\setminus U$ is non-empty independent set  with  $a_1\in N(Z)$.  So  every subset  $W\subset V\setminus(N(Z)\cup Z)$ can not contain  an odd cycle. Hence  $U=V$ is the  unique set  such that $\calD^*(U)\not=\emptyset, \calD^*(U)\not=\{\emptyset\}.$ As a consequence we have that $\Ass (I^{i}_G)=\Ass (I_G)\cup \{\m\} $ for $i\geq 2.$ Hence $\astab(I_G)=\dstab(I_G)=2$.
\begin{figure}[!ht] 
\centering
\includegraphics[height=6cm,width=8cm]{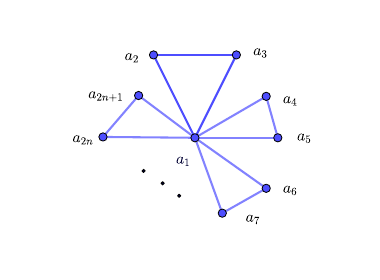}
\caption{Union of triangles } \label{Fig.6-0}
\end{figure} 
\end{ex}

\begin{ex}
Let $G_1=(V_1,E_1)$ be a graph as in the Figure \ref{Fig.6} with the vertex set  $V_1=\{a,b,c,d,e,f,g,h,i,j,k,l,m,n\}$. Then the graph $G$ of Figure  \ref{Fig.1} is $G_1$-dominant* and gives the minimum in Theorem \ref{astab-main}. So we have $\dstab(I_{G_1})=6$.
\begin{figure}[!ht]
\centering
\includegraphics[height=6cm,width=8cm]{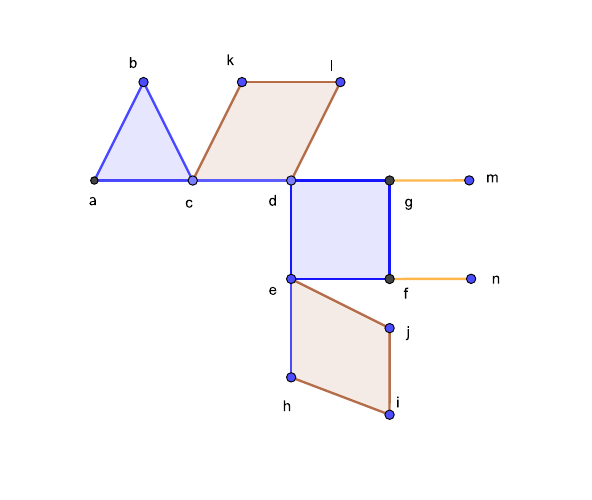}
\caption{$G$ and $G_1$ } \label{Fig.6}
\end{figure} 
\noindent The irreducible component $(a^2, b^2, c^3, d^3, e^3, f^2, g^2, h^2, i, j, k, l, m, n)$ appears in the irreducible decomposition of $I_{G_1}^6$.
\end{ex}

\begin{ex}
Let $G=(V,E)$ be the graph with the vertex set $V=\{a_1,a_2,\ldots,a_{13}\}$ and $E=\{a_{i}a_{i+1}\mid i=1,\ldots,10\}\cup \{a_{1}a_{11},a_{1}a_{12},a_{1}a_{13},a_{12}a_{13} \}.$
Note that $G$ is the union of a triangle  and a hendecagon    represented in Figure  7.\\
Let $W=\{a_1,a_2,\ldots,a_{10},a_{11}\} $ and $G_{W}$ be the induced subgraph of $G$. We can see that  $G_{W}$  is factor-critical and $G$-dominant*  with $\nu(G_{W})=5$.   We have by Theorem \ref{astab-main}  that $\dstab(I_G)=6$.\\
Now we look for embedded associated primes strictly contained in the maximal ideal $\m ^{{\bf 1}_{V}}$. By Theorem \ref{astab-main}, we should look for non-maximal independent sets $Z\subset V$ such that $\calD^*(U)\not=\emptyset$ where $U=V\setminus Z$, that is  there exists a set $W\subset V\setminus (Z\cup N(Z))$  for which every connected component of $G_W$ contains an odd  cycle. In our example, this implies that $W$ is connected and $Z\subset \{a_3,a_4,\ldots,a_{10}\}$. Since  $\astab(I_G, U)$ is determined  by the number of edges in $W$, we have  $\astab(I_G, U)$ will be  the largest possible if $Z$ has cardinality one. We can examine all $8$ cases where $Z_i=\{a_{i}\}$, for $i=3,\ldots,10.$ \\
Indeed, let $U_i=V\setminus Z_i$. We will see by Theorem \ref{astab-main} that $\astab(I_G, U_i)=7$ for $i=4,\dots,9$ and $\astab(I_G, U_i)=8$ for $i=3,10$.  For instance we consider the cases $Z_6=\{a_{6}\}, Z_3=\{a_{3}\} $ since the other cases are similar.
\begin{enumerate}
	\item If $Z_6=\{a_{6}\}$ then  $N(Z_6)=\{a_{5}, a_{7}\}$. So  $W_6=\{a_1,a_2,a_3,a_9,a_{10},a_{11},a_{12},a_{13}\} $ is  $U_6$-dominant*   and we have by Theorem \ref{astab-main} that  $\astab(I_G, U_6)=7$.
	\item If $Z_3=\{a_{3}\}$ then similarly  we have $N(Z_3)=\{a_{2}, a_{4}\}$. We can see that the set     $W_3=\{a_1,a_6,a_7,a_8,a_9,a_{10},a_{11},a_{12},a_{13}\} $ is  $U_3$-dominant* and we also have by Theorem \ref{astab-main} that $\astab(I_G, U_3)=8$. 
\end{enumerate}
Finally we get  $\astab(I_G)=8$. 
\begin{figure}[h!tbp]
\centering
\includegraphics[height=6cm,width=8cm]{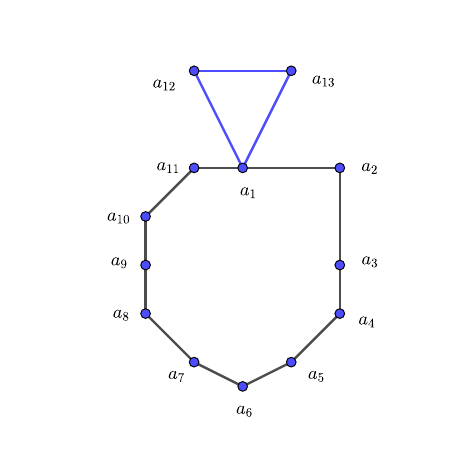}
\caption{graph $G$ } \label{Fig.7}
\end{figure}
\begin{figure}[h!tbp]
\centering
\includegraphics[height=6cm,width=12cm]{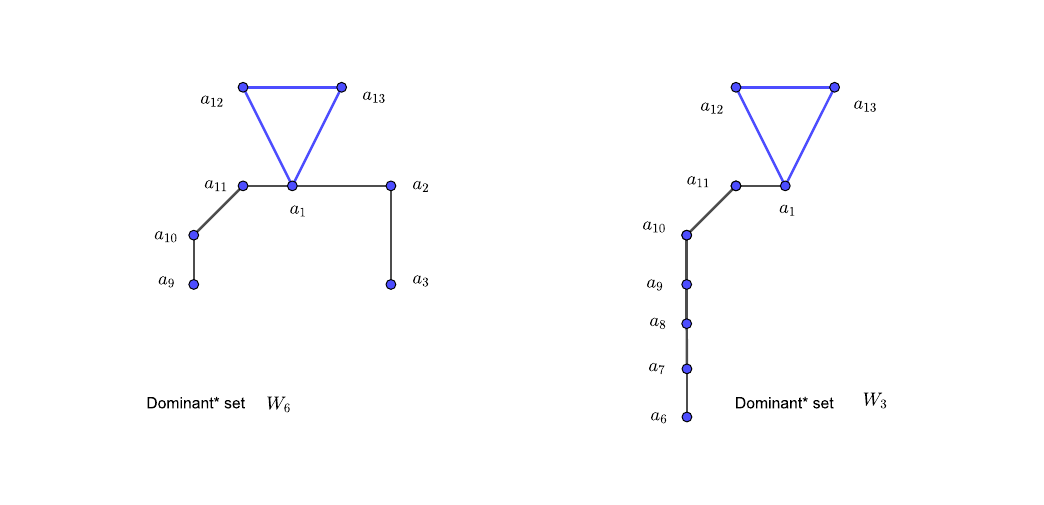}
\caption{Dominants* sets} \label{Fig.8}
\end{figure}
\end{ex}

\medskip
Now we can do a more general case.

\begin{ex}
Let $G=(V,E)$ be a graph with the vertex set $V=\{a_1,a_2,\ldots,a_{2l+3}\}$ and $E=\{a_{i}a_{i+1}\mid i=1,\ldots,2l\}\cup \{a_{1}a_{2l+1},a_{1}a_{2l+2},a_{1}a_{2l+3},a_{2l+1}a_{2l+3} \}.$ 
Then   $\dstab(I_G)=l+1, \astab(I_G)=2l-2.$\\
 From figure  9,  we can see that $G$ is the union of a triangle  and a polygon with $2l+1$ vertices.  
Let $W=\{a_1,a_2,\ldots,a_{2l+1}\} $ and $G_{W}$ be the induced subgraph of $G$. Then $G_{W}$  is factor-critical and $G$-dominant*, with $\nu(G_{W})=l$. By Theorem \ref{astab-main} we have  
 $\dstab(I_G)=l+1$.\\ 
Now we look for embedded associated primes strictly contained in the maximal ideal $\m ^{{\bf 1}_{V}}$. By the theorem \ref{astab-main} we should look for non-maximal independent sets $Z\subset V$  such that $\calD^*(U)\not=\emptyset$ where $U=V\setminus Z$, that is  there exists a set $W\subset V\setminus (Z\cup N(Z))$  for which every connected component of $G_W$ contains an odd cycle, in our example this implies that $W$ is connected and $Z\subset \{a_3,a_4,\ldots,a_{2l}\}$. $\astab(I_G, U)$ is determined  by the number of edges in $W$, so that  $\astab(I_G, U)$ will be  the largest possible when $Z$ has cardinal one. Let examine all cases.\\
Let $Z_i=\{a_{i}\}, U_i=V\setminus Z_i$, we will see by theorem \ref{astab-main} that $\astab(I_G, U_i)=2l-3$ for $i=4,\dots,{2l-1}$ and $\astab(I_G, U_i)=2l-2$ for $i=3,2l$.\\
To be more precise let consider $Z_{l+1}=\{a_{l+1}\}$ and $ Z_3=\{a_{3}\}$, the other cases are similar.\\  We have $N(Z_{l+1})=\{a_{l}, a_{l+2}\}$. So that $W_{l+1}=\{a_1,\ldots,a_{l-2},a_{l+4},\ldots,a_{2l+3}\} $ is   $U_{l+1}$-dominant*, so by theorem \ref{astab-main} we have $\astab(I_G, U_{l+1})=2l-3$.\\ 
Consider $Z_3=\{a_{3}\}$, we have $N(Z_3)=\{a_{2}, a_{4}\}$. So that $W_3=\{a_1,a_6,\ldots,a_{2l+3}\} $ is  $U_3$-dominant*, so by theorem \ref{astab-main} we have $\astab(I_G, U_3)=2l-2$. Finally we get  $\astab(I_G)=2l-2$.
\begin{figure}[h!tbp]
\centering
\includegraphics[height=6cm,width=8cm]{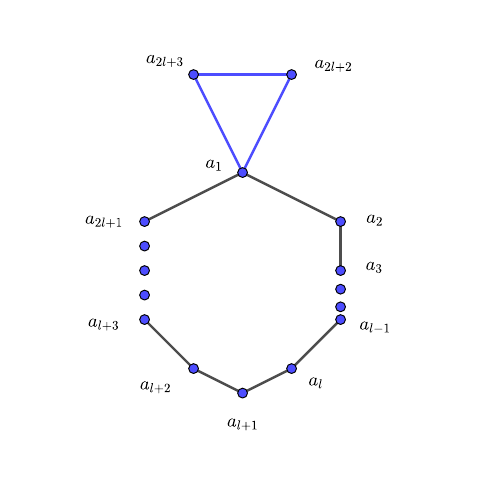}
\caption{Graph $G$ with $2l+3$vertices  } \label{Fig.9}
\end{figure}
\begin{figure}[h!tbp]
\centering
\includegraphics[height=6cm,width=12cm]{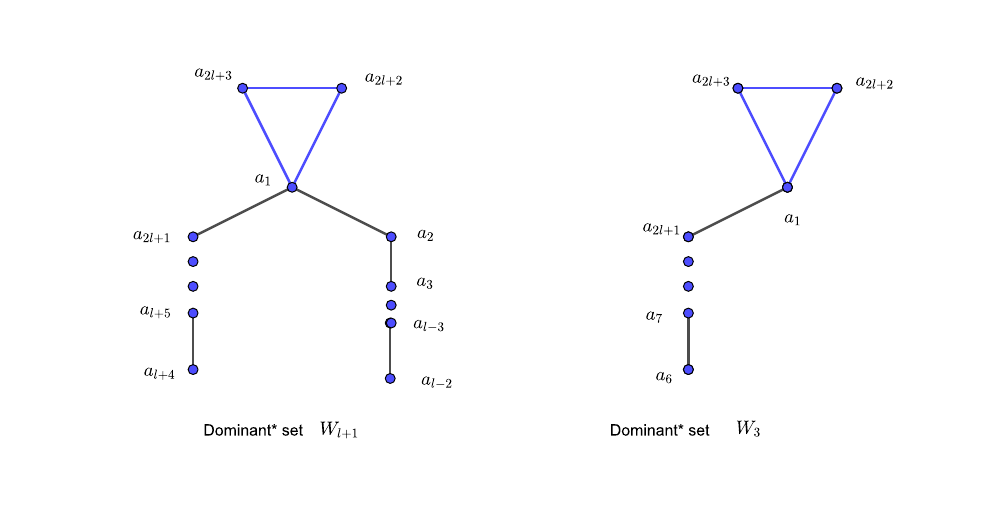}
\caption{Dominants* sets} \label{Fig.10}
\end{figure}
\end{ex} 
\section{$G$ has an optimal ear decomposition starting with an odd cycle.}
The following Lemma is due to Frank \cite{F}.
\begin{lem}\label{l-frank} For every edge $f$ of a 2-edge connected graph $G$ there exists an optimal ear decomposition of $G$ such that the first ear uses $f$.
\end{lem}
\begin{thm}\label{t-equal-phi} Let $G$ be a 2-edge connected non bipartite graph. Then there exists an optimal ear decomposition of $G$ starting with an odd cycle. As a consequence $\varphi(G)_{odd} =\varphi (G)$. 
\end{thm}
\begin{proof}The proof will be by induction on the number of vertices of $G$. Suppose $\card(V(G))=3$ then  $G=K_3$, our claim is true. 
Suppose $\card(V(G))=d$,  $d\geq 4$. By induction hypothesis our claim is true for all 2-edge connected graphs $G'$ with $\card(V(G'))<d$.  Let $F_1,F_2,\ldots,F_{s}$ be an optimal ear decomposition of $G$ where $F_1$ is a cycle.  For  $i=1,\ldots,s$ let $G_i$ be the graph with ear decomposition $F_1,\ldots,F_i$, note that by Withney's theorem \ref{t-Withney} $G_i$ is a 2-edge connected graph. If $F'_1,\ldots,F'_{i}$ is an optimal ear decomposition of $G_i$ then $F'_1,\ldots,F'_{i},F_{i+1},\ldots,F_{s}$ is an optimal ear decomposition of $G$. Let $\chi _i$ denote the number of even ears among $F_{i+1},\ldots,F_{s}$, we have  $\varphi (G_i)=\varphi (G)- \chi_i$.\\  If $F_1$ is odd our claim is done. So suppose that $F_1$ is an even cycle, let us recall that an even ccle is bipartite. Since $G$ is non bipartite there exists a small integer $j>1$ such that $G_j$ is non bipartite. If $\card(V(G_j))<d$ then by the induction hypothesis $G_j$ has an optimal  ear decomposition starting with an odd cycle $F'_1,\ldots,F'_{j}$. Hence $F'_1,\ldots,F'_{j},F_{j+1},\ldots,F_s$ is an  optimal  ear decomposition of $G$ starting with an odd cycle.\\
 If $\card(V(G_j))=d$ then  $G_{j-1}$ is a bipartite graph with a partition $A,B$. We have two cases either $F_{j}$ has one or more edges. \\
 \begin{enumerate}
\item If $F_{j}$ has one edge $f$, then $V(G_{j-1})=V(G), E(G_j)=E(G_{j-1})\cup \{f\}$ where $f=PQ, P,Q\in A\cup B$. If $P\in A, Q\in B$ then $G_j$ is bipartite, hence we can assume that $P,Q\in A$. By  lemma  \ref{l-frank} there exists an optimal ear decomposition of $G_j$: $F'_1,\ldots,F'_{j}$ starting with a cycle $F'_1$ that uses $f$. All edges of $F'_1$ except $f$ belongs to the bipartite graph $G_{j-1}$, but any path in a bipartite graph between  $P,Q\in A$ is even hence $F'_1$ is an odd cycle. It follows that $F'_1,\ldots,F'_{j},F_{j+1},\ldots,F_{s}$ is an optimal ear decomposition of $G$ starting with an odd cycle.
\item  If $F_{j}=P_{i_1}P_{i_2}\ldots P_{i_l}$ with $l\geq 3$, then $V(G_{j-1})=V(G), E(G)=E(G_{j-1})\cup \{P_{i_1}P_{i_2},P_{i_2}P_{i_3},\ldots,P_{i_{l-1}}P_{i_l}\}$ and $P_{i_1},P_{i_l}\in A\cup B,  \deg_{G_j}(P_t)=2$  for $1<t<l$. We have four cases.
   \begin{enumerate}
\item If $P_{i_1}\in A, P_{i_l}\in B$ and $l$ is even then we set $ A'=A\cup \{P_{i_1},P_{i_3},  \ldots,P_{i_{l-1}}\}, B'=B\cup \{P_{i_2}, \ldots,P_{i_{l}}\}$ hence  $G_j$ is bipartite with partition $A',B'$, a contradiction. 
\item If $P_{i_1}\in A, P_{i_l}\in B$ and $l$ is odd then we set $ A'=A\cup \{P_{i_j}, j\in 2\fn\}, B'=B\cup \{P_{i_j}, j\notin 2\fn, j\not=1\}$. Let $G'$ be the graph with  $V(G')=V(G_j), E(G')=E(G_j)\setminus \{ P_{i_1}P_{i_2}\}$,  hence  $G'$ is bipartite with partition $A',B'$. By  lemma  \ref{l-frank} there exists an optimal ear decomposition of $G_j$ starting with a cycle $C$ that uses $f=P_{i_1}P_{i_2}$. All edges of $C$ except $f$ belongs to the bipartite graph $G'$, but any path in the bipartite graph $G'$ between  $P_{i_1}, P_{i_2}\in A'$ is even hence $C$ is an odd cycle. Now as before we add the ears $F_{j+1},\ldots,F_{s}$ to the optimal ear decomposition of $G_j$, so that we get  an optimal ear decomposition of $G$ starting with an odd cycle.
\item If $P_{i_1},P_{i_l}\in A$ and $l$ is odd then we set $ A'=A\cup \{P_{i_1},P_{i_3},  \ldots,P_{i_{l}}\}, B'=B\cup \{P_{i_2}, \ldots,P_{i_{l-1}}\}$ hence  $G_j$ is bipartite with partition $A',B'$, a contradiction. 

\item If $P_{i_1},P_{i_l}\in A$ and $l$ is even then we set $ A'=A\cup \{P_{i_j}, j\in 2\fn\}, B'=B\cup \{P_{i_j}, j\notin 2\fn, j\not=1\}$. Let $G'$ be the graph with  $V(G')=V(G_j), E(G')=E(G_j)\setminus \{ P_{i_1}P_{i_2}\}$,  hence $G'$ is bipartite with partition $A',B'$. By  lemma  \ref{l-frank} there exists an optimal ear decomposition of $G_j$ starting with a cycle $C$ that uses $f=P_{i_1}P_{i_2}$. All edges of $C$ except $f$ belongs to the bipartite graph $G'$, but any path in the bipartite graph $G'$ between  $P_{i_1}, P_{i_2}\in A'$ is even hence $C$ is an odd cycle. Now as before we add the ears $F_{j+1},\ldots,F_{s}$ so that we get  an optimal ear decomposition of $G$ starting with an odd cycle.
 \end{enumerate} \end{enumerate}
\end{proof}

\end{document}